\documentclass{amsart}

\usepackage{amssymb}
\usepackage{colonequals}
\usepackage{fullpage}
\usepackage{xcolor,paralist,hyperref,fancyhdr,etoolbox}
\usepackage{mathtools}

\newtheorem{thm}{Theorem}[]
\numberwithin{thm}{section}
\newtheorem{defn}[thm]{Definition}
\newtheorem{cor}[thm]{Corollary}
\newtheorem{lem}[thm]{Lemma}
\newtheorem{prop}[thm]{Proposition}
\newtheorem{rmk}[thm]{Remark}

\newcommand{\Fq}{\mathbb{F}_q}
\newcommand{\Pp}{\mathbb{P}}
\newcommand{\Ovd}{\mathcal{O}}

\newcommand{\PG}{\operatorname{PG}}
\newcommand{\ltf}{line-transverse-free}
\newcommand{\ptf}{plane-transverse-free}
\newcommand{\ktf}{$k$-flat-transverse-free}
\newcommand{\ctf}{circle-transverse-free}
\newcommand{\spf}{space-filling}
\newcommand{\SSFS}{\operatorname{SSFS}}

\title{Transverse-freeness in Finite Geometries}

\author{Charlie Bruggemann}
\address{Reed College}
\email{cabruggemann@reed.edu}

\author{Vera Choi}
\address{Tufts University}
\email{vera.choi@tufts.edu}

\author{Brian Freidin}
\address{Auburn University}
\email{bgf0012@auburn.edu}

\author{Jaedon Whyte}
\address{Massachusetts Institute of Technolog}
\email{jnw@mit.edu}

\subjclass[2020]{Primary 14N25; Secondary 14N20, 51E21, 51E23}
\keywords{transverse-free, finite fields, projective space, inversive plane, hyperbolic plane}

\begin{document}

\begin{abstract}
    We study projective curves and hypersurfaces defined over a finite field that are tangent to every member of a class of low-degree varieties. Extending 2-dimensional work of Asgarli, we first explore the lowest degrees attainable by smooth hypersurfaces in $n$-dimensional projective space that are tangent to every $k$-dimensional subspace, for some value of $n$ and $k$. We then study projective surfaces that serve as models of finite inversive and hyperbolic planes, finite analogs of spherical and hyperbolic geometries. In these surfaces we construct curves tangent to each of the lowest degree curves defined over the base field.
\end{abstract}

\maketitle

\section{Introduction}

A classical theorem of Bertini implies that a smooth projective curve defined over an infinite field admits a transverse line, i.e. one that is not tangent to the curve. In \cite{poonen}, Poonen showed that this result fails to hold when everything is defined over a finite field $\Fq$. That is, there are transverse-free curves (ie. curves tangent to every line defined over the (finite) base field) of any sufficiently high degree, c.f. Example 2.3 of \cite{transverse-free}. In \cite{transverse-free}, Asgarli also establishes a tight lower bound on the degrees of such curves, showing that the degree of a smooth transverse-free plane curve is at least $q+2$. Later, in \cite{freidin-asgarli}, Asgarli and Freidin employed combinatorial and linear algebraic techniques to describe the set of all transverse-free curves.

In this paper, we extend the study of transverse-free objects to other settings. These include notions of transverse-free hypersurfaces in higher dimensional projective spaces, as well as curves within a specific projective surface, with or without an additional symmetry. In the context of projective hypersurfaces we explore the different properties of various notions of transverse-freeness, as well as show that lower-dimensional examples can be easily extended to give higher-dimensional examples. In the context of curves within a surface we exploit the algebraic descriptions of curves and tangencies within surfaces in order to demonstrate the existence of low-degree transverse-free curves.

Section~\ref{sec:prelim} is devoted to defining the spaces, along with symmetries, and the objects and properties of study therein. The spaces include projective spaces, a quadric surface $\Ovd$ in the 3-dimensional projective space, as well as an involution $*$ preserving $\Ovd$. The intersection of $\Ovd$ with planes of $\Pp^3$ are called circles. Following \cite{dembowski} and \cite{tits}, the rational points and circles in $\Ovd$ have the incidence structure of a finite inversive plane, an analog of spherical geometry. When $q$ is even, following \cite{hyperbcrowe}, some rational points and $*$-invariant circles have the incidence structure of a finite hyperbolic plane, an analog of hyperbolic geometry.

In projective spaces we will describe hypersurfaces and tangency with $k$-dimensional linear subspaces, which we will refer to as $k$-flats. Within the surface $\mathcal{O}$ we will describe curves and tangency with circles. With all of these definitions we will be able to define the transverse-free objects  that are the focus of this work. In the projective space these are hypersurfaces tangent to every $k$-flat, for some $k$, and in $\Ovd$ these are curves tangent to every circle, or every $*$-invariant circle if we are discussing hyperbolic curves.

In Section~\ref{ndimproj} we study hypersurfaces in higher dimensional projective spaces. Here the notions of being \ktf\ for different $k$ do not always coincide. In particular, we study smooth, \spf\ hypersurfaces (c.f. Definition~\ref{spf}). Such objects only exist with degrees $d\geq q+1$. For hypersurfaces of degree $d=q+1$ and $d=q+2$, we study when these \spf\ hypersurfaces are \ktf\, and for which $k$. Our main results in this section can be summarized as follows.

\begin{thm}[cf. \ref{qp1NeverLTF}, \ref{qp1AlwaysPTF}, \ref{qp2AlwaysLTF}, \ref{qp2NeverPTF}]
    No degree $q+1$ smooth, \spf\ surface in $\Pp^3$ is \ltf, while for odd $n$ and even $k$, every degree $q+1$ smooth \spf\ hypersurface in $\Pp^n$ is \ktf. Every degree $q+2$ smooth, \spf\ hypersurface in $\Pp^n$ is \ltf, while  for odd $q$ no degree $q+2$ smooth, \spf\ hypersurface in $\Pp^3$ is \ptf.
\end{thm}

Compare the above to \cite{asgarli-lines} and \cite{asgarli-subspaces}, which establish transversals to hypersurfaces of sufficiently low degree compared with $q$. While there is still some gap between our degree bounds and theirs, we also only address smooth, \spf\ hypersurfaces. It is possibly important, then, that \cite{asgarli-lines} notes that smooth degree $q+2$ \spf\ hypersurfaces in $\Pp^3$ are not known to exist for all $q$.

Then in Section~\ref{sec:cones} we construct \ktf\ hypersurfaces in high dimensional projective spaces as cones over \ktf\ hypersurfaces in lower-dimensional projective spaces. While these cones are not smooth, we can state very precisely where the singularities occur. In particular, extending a smooth hypersurface in $\Pp^n$ to $\Pp^{n+1}$ creates exactly one singular point. In this way we can construct many examples of hypersurfaces satisfying various notions of transverse-freeness from known examples. In particular starting from transverse-free curves in the plane produces many \ltf\ hypersurfaces in arbitrary dimensions.

Finally we construct transverse-free curves within the surface $\Ovd$. In section~\ref{sec:ctf} we construct curves tangent to every circle of $\Ovd$, which we call \ctf\ inversive curves, and in Section~\ref{sec:hyp} we construct curves invariant under $*$ that are tangent to every $*$-invariant circle, which we call \ltf\ hyperbolic curves. In each case one approach to construct transverse-free curves follows the strategy of \cite{freidin-asgarli}. We pair each circle in our ambient surface with a distinct point at which a curve can be made tangent to that circle. While the curves constructed by this approach are not guaranteed to be smooth, we hope they at least have no singularities at rational points. This approach relies on the combinatorics of partial spreads in the incidence structure of the finite projective space, in particular pairwise disjoint sets of lines that are all tangent to a fixed surface as constructed in \cite{spread}. Our results here can be summarized as follows.

\begin{thm}[cf. Theorems~\ref{linctf}, \ref{hypltf}]
    Over $\Fq$, there exist polynomials of all degrees $d>-1+\sqrt{2(q^2+1)}$ defining transverse-free curves in $\Ovd$. There also exist $*$-invariant polynomials of all degrees $d>-3/2+\sqrt{14q^2+10q-95/4}$ defining transverse-free hyperbolic curves.
\end{thm} 

Another approach to construct transverse-free curves is to cosndier curves whith singularities that form a blocking set, a set of points that intersects every circle or line in the space. This approach requires results on blocking sets in finite inversive planes from \cite{blocking} and our hyperbolic generalizations thereof. Our results along these lines are as follows.

\begin{thm}[cf. Theorems~\ref{invsingdeg}, \ref{hypsingdeg}]
    For all $d>-1 + \sqrt{3\left(q^2+1-\sqrt[q+1]{\binom{q^2}{q}(q-1)!}\right)}$ there is a degree $d$ polynomial defining a transverse-free curve in $\Ovd$. And for all $d>-3+\sqrt{12q^2-12q+1-24\sqrt[q/2]{\binom{q(q-1)/2}{q/2}\frac{(q/2)!}{q^2-1}}}$ there is a degree $d$, $*$-invariant polynomial defining a transverse-free hyperbolic curve.
\end{thm}

While these last results may be hard to parse, Remarks~\ref{inv-simp} and \ref{hyp-simp} give the simpler-looking bounds of $d>-1+3\sqrt{q\ln(q)}$ and $d>-3+\sqrt{1+48q\ln(q)}$, respectively.

With the linear algebra we use to derive the results above, a curve with some number of tangencies can always be constructed with lower degree than a curve with the same number of singularities. On the other hand, a transverse-free curve may require fewer singularities than tangency conditions. The result of this balancing act is that the singular curves of Theorems~\ref{invsingdeg} and \ref{hypsingdeg} occur with lower degrees than the potentially non-singular curves of Theorems~\ref{linctf} and \ref{hypltf}. Moreover the degrees of our transverse-free curves are still bounded by linear functions. This is consistent with the degree $q+2$ bounds of \cite{transverse-free} in the projective plane, even while our notions of transverse-freeness are different in each setting.

\subsection*{Acknowledgements} This material is based upon work supported by the National Science Foundation under Grant Number DMS-1851842, and conducted at the MathILy-EST 2023 REU under the supervision of the third author. Any opinions, findings, and conclusions or recommendations expressed in this material are those of the author(s) and do not necessarily reflect the views of the National Science Foundation.

\section{Preliminaries}\label{sec:prelim}

For the rest of this paper, fix a prime power $q$ and the corresponding finite field $\Fq$. We begin by describing the ambient spaces in which our objects of study can be found. We will focus on algebraic descriptions of these spaces, pointing out the combinatorial properties and connections to finite geometries as they come up.

Each space will be described ultimately in terms of vectors in some vector space $\overline{\Fq}^k$ over the algebraic closure of the field $\Fq$. We will often distinguish $\Fq$-points (those that can be realized with coordinates in $\Fq$) from non-$\Fq$ points (those with some coordinate in a non-trivial extension of $\Fq$) . Other objects (lines, planes, curves, etc.) will also be treated algebraically in terms of defining polynomials. We will only deal with polynomials defined over $\Fq$ (i.e. by polynomials with coefficients in $\Fq$) but we still understand them to contain non-$\Fq$ points. To that end, the notation $X(\Fq)$ denotes just the $\Fq$-points of the object $X$.

There is rich algebraic geometry in the connections between the polynomials we will study and geometric objects. We will, however, phrase everything in terms of the polynomials themselves.

\subsection{Projective space}

The $n$-dimensional projective space $\Pp^n$ is the quotient $\big(\overline{\Fq}^{n+1}\backslash\{0\}\big)/\overline{\Fq}^\times$, where $\overline{\Fq}^\times$ acts by scaling. We may represent points in $\overline{\Fq}^{n+1}$ by tuples $(x_0,x_1,\ldots,x_n)$, or their images in $\Pp^n$ with colons as $(x_0:x_1,\cdots:x_n)$.

For a non-constant homogeneous polynomial $f\in\Fq[x_0,\ldots,x_n]$, we will refer to the hypersurface $S_f$ as the image in $\Pp^n$ of the set
\[
    \{P\in\overline{\Fq}^{n+1} \mid f(P)=0\}.
\]
For $\lambda\neq 0$ the hypersurfaces $S_f$ and $S_{\lambda f}$ coincide, so we treat the polynomials $f$ and $\lambda f$ as equivalent. If the polynomial $f$ has degree $d$, we say $S_f$ \textbf{has degree $d$}, and we call $S_f$ a \textbf{degree-$d$ hypersurface}. In $\Pp^3$ we will refer to a hypersurface $S_f$ as a surface. And in $\Pp^2$ we will refer to a hypersurface $S_f$ as a projective curve, to distinguish from inversive and hyperbolic curves defined below.

If $f$ is a linear function, we call the hypersurface $S_f$ a \textbf{hyperplane}. The next simplest kind of object we will intersect with hypersurfaces are what we call $k$-flats, which immediately generalize hyperplanes.
\begin{defn}[$k$-flat]
    For $1\leq k\leq n-1$, a $k$-flat $\pi\subset\Pp^n$ is the intersection of $n-k$ hyperplanes whose defining linear polynomials $L_1,\ldots,L_{n-k}$ are linearly independent in $\Fq[x_0,\ldots,x_n]$.
\end{defn}

When $k=1$ a $k$-flat is simply called a line. When $k=2$ a $k$-flat is simply called a plane. Recall that all functions, and thus all $k$-flats, are defined over $\Fq$.

\begin{rmk}
    The $\Fq$-points, $\Pp^n(\Fq)$, together with the lines of $\Pp^n$ as defined above, have the incidence structure of a finite projective space, known as $\PG(n, q)$.
\end{rmk}

We say a hypersuface $S_f$ is \textbf{smooth} at a point $P\in S_f$ corresponding to $\tilde{P}\in\overline{\Fq}^{n+1}$ if the gradient
\[
    \nabla f(\tilde{P}) = \left(\frac{\partial f}{\partial x_0}(\tilde{P}),\ldots,\frac{\partial f}{\partial x_n}(\tilde{P})\right)
\]
does not vanish. Note that for homogeneous $f$, each $\frac{\partial f}{\partial x_j}$ is also homogeneous, so $\nabla f(\tilde{P})=0$ if and only if $\nabla f(\lambda\tilde{P})=0$ for $\lambda\neq 0$, so the choice of $\tilde{P}$ does not matter. In case $S_f$ is not smooth at $P$, we say $S_f$ is \textbf{singular} at $P$, or that $P$ is a \textbf{singularity} of $S_f$. We say $S_f$ is smooth if it is smooth at every $P\in S_f$.

When a hypersurface is smooth at a point, we can define its tangent space at that point as a hyperplane.

\begin{defn}[Tangent space]
    Let $S=S_f\subset\Pp^n$ be a hypersurface defined by a homogeneous polynomial $f$ and let $P\in S$ be a point. Suppose $\tilde{P}\in\overline{\Fq}^{n+1}$ corresponds to $P\in\Pp^n$. If $S_f$ is smooth at $P$ then the tangent space to $S$ at $P$, denoted $T_PS$, is the hyperplane $\pi\subset\Pp^n$ defined by the linear function
    \[
        \frac{\partial f}{\partial x_0}(P)x_0 + \cdots + \frac{\partial f}{\partial x_n}(P)x_n = 0.
    \]
    
    Note for $\lambda\neq 0$, that $\nabla f(\lambda\tilde{P})$ is a scalar multiple of $\nabla f(\tilde{P})$, so the functions listed above for $\tilde{P}$ and for $\lambda\tilde{P}$ define the same hyperplane.

    If $S_f$ is singular at $P\in S_f$ we define $T_PS_f = \Pp^n$.
\end{defn}

In general for a $k$-flat $\pi\subset\Pp^n$, we say $S_f$ is tangent to $\pi$ at $P$ if $\pi\subset T_PS_f$. More generally, we will say two smooth hypersurfaces, $S_f$ and $S_g$, are tangent at $P\in\Pp^n$ if $P\in S_f\cap S_g$ and $T_PS_f=T_PS_g$. This is equivalent to saying $f(P)=g(P)=0$ and there is a non-zero scalar $\lambda\in\Fq$ so that $\nabla f(P)=\lambda\nabla g(P)$.

Since all the computations we do in coordinates (e.g. smoothness and tagngent plane above) do not depend on the point $\tilde{P}\in\overline{\Fq}^{n+1}$ we choose corresponding to a point $P\in\Pp^n$, we will often ease notation by simply referring to $f(P)$ or $\nabla f(P)$, even though $f$ and $\nabla f$ are really only defined on $\overline{\Fq}^{n+1}$.

It will be useful later on to have some information about duality in projective spaces.

\begin{defn}[Dual]
    The dual of $\Pp^n$ is the projective space $(\Pp^n)^*$ where the hyperplane $\pi\subset\Pp^n$ defined by the function $A_0x_0+A_1x_1+\cdots+A_nx_n$ corresponds to a point $\pi^*=(A_0:A_1:\ldots:A_n)\in(\Pp^n)^*$ and a point $P=(P_0:P_1:\cdots:P_n)\in\Pp^n$ corresponds to a hyperplane $P^*\subset(\Pp^n)^*$ defined by the function $P_0x_0+P_1x_1+\cdots+P_nx_n$.
\end{defn}

In general, a $k$-flat $\pi\subset\Pp^n$ cut out by the hyperplanes $\pi_i$, $1\leq i\leq n-k$, corresponds to the unique $(n-k)$-flat $\pi^*\subset(\Pp^n)^*$ containing the points $\pi^*_1,\ldots,\pi^*_{n-k}$. One can thus see that for flats $\pi_1,\pi_2\subset\Pp^n$, $\pi_1\subset\pi_2$ if and only if $\pi^*_1\supset\pi^*_2$.

Let $\pi_q(k,n)$ be the number of $k$-flats in $\Pp^n$ defined over $\Fq$. By counting the number of linearly independent ordered $k$-tuples in $\Fq^{n+1}$, and dividing by the number of ordered bases in $\Fq^{k+1}$, one sees that
\[
    \pi_q(k,n) = \prod_{i=0}^k\frac{q^{n+1-i}-1}{q^{k+1-i}-1}.
\]

Exploiting the relationship between the incidence structures of $\Pp^n$ and $(\Pp^n)^*$ we also have the following result.

\begin{lem}\label{duallem}
    In $\Pp^n$, the number of $k$-flats that contain a fixed $\ell$-flat $\pi_0$ is equal to $\pi_q(n-k,n-\ell)$
\end{lem}

\begin{proof}
    A $k$-flat $\pi$ of $\Pp^n$ contains the $\ell$-flat $\pi_0$ if and only if the $(n-k)$-flat $\pi^*$ is contained in the $(n-\ell)$-flat $\pi^*_0$ within $(\Pp^n)^*$. Thus the number of $k$-flats containing $\pi_0$ is precisely the number of $(n-k)$-flats contained within $\pi^*_0$. Since $\pi^*_0$ is isomorphic to $\Pp^{n-\ell}$, the number of $(n-k)$ flats it contains is precisely $\pi_q(n-k,n-\ell)$.
\end{proof}

\subsection{Quadric surfaces and Inversive planes}\label{sec:inv}

Let $\gamma\in\Fq$ be such that $t^2+t+\gamma$ is irreducible over $\Fq$. Let $\Ovd$ be the surface $\Ovd=S_\varphi\subset\Pp^3$ defined by the polynomial
\[
    \varphi(x,y,z,w) = xw-y^2-yz-\gamma z^2.
\]
Note that $\nabla\varphi(x,y,z,w) = (w,-2y-z,-y-2\gamma z,x)$ does not vanish at any non-zero vector $(x,y,z,w)$, so $\Ovd$ is a smooth surface in $\Pp^3$.

Given a polynomial $f\in\Fq[x,y,z,w]$, we refer to the curve $C_f\subset\Ovd$ as the intersection of the surface $S_f$ with $\Ovd$, namely
\[
    C_f = S_f\cap\Ovd.
\]
If $\pi$ is a plane in $\Pp^3$ not tangent to $\Ovd$, i.e. $\pi=S_\ell$ for a linear function $\ell$, then we call the curve $C_\ell$ a \textbf{circle}.

The points of $\Ovd$ can also be described parametrically. In addition to $(1:0:0:0)$, the remaining points all have the form
\[
    (s^2+st+\gamma t^2:s:t:1)
\]
for $(s,t)\in\overline{\Fq}^2$.

From \cite{tits} the parametric description of $\Ovd$ above verifies it as an elliptic quadric surface, an ovoid. In particular, the irreducibility of $t^2+t+\gamma$ verifies that no three points of $\Ovd$ lie on an $\Fq$-line, so every line of $\Pp^3$ intersects $\Ovd$ in either 0, 1, or 2 points. According to \cite{dembowski}, the $\Fq$-points of $\Ovd(\Fq)$, together with the circles, form an \textbf{inversive plane}, a finite analog of spherical geometry.

The axioms and properties of finite inversive planes that we will find most useful are summarized here. See e.g. \cite{dembowski} for more details.
\begin{itemize}
    \item Every three points of $\Ovd(\Fq)$ lie in exactly one circle.
    \item If $C$ is a circle containing $P\in\Ovd(\Fq)$ but not $Q\in\Ovd(\Fq)$, there is exactly one circle $C'$ containing $P$ and tangent to $C$ at $Q$.
    \item $\Ovd(\Fq)$ contains $q^2+q$ points and $q(q^2+1)$ circles.
    \item Every circle contains exactly $q+1$ $\Fq$-points., $q(q+1)$ circles pass through each $\Fq$-point, and $q+1$ circles pass through each pair of $\Fq$-points.
\end{itemize}

We will thus refer to curves in $\Ovd$ as \textbf{inversive curves}.

We say the curve $C_f$ is smooth at a point $P\in C_f$ if $S_f$ is smooth at $P$ and $T_P S_f \neq T_P\Ovd$, or more simply put, if $T_P\Ovd\not\subset T_PS_f$. We say $C_f$ is smooth if it is smooth at every $P\in C_f$. With this we can also define tangent lines to smooth curves as projective lines tangnet to $\Ovd$.

\begin{defn}[Tangent line]\label{tangline}
    Let $C = C_f\subset\Ovd$ be a curve and $P\in C$ a point. If $C$ is smooth at $P$ then the tangent line to $C$ at $P$ is the line
    \[
        T_PC \colonequals T_PS_f\cap T_P\Ovd \subset\Pp^3.
    \]
    
    If $C$ is not smooth at $P$, we still define $T_PC = T_PS_f\cap T_P\Ovd$, but in this case it is just $T_P\Ovd$.
\end{defn}

For a line $L\subset T_P\Ovd$, we say a curve $C_f$ is tangent to $L$ at $P$ if $L\subset T_PC_f$. In particular this applies to the tangent line of a circle. In contrast to $\Pp^n$ where distinct hyperplanes passing have distinct tangent spaces, since each hyperplane is its own tangent space, we will show that many circles in $\Ovd$ can share a tangent line.

\begin{lem}\label{circtangency}
    Two circles, $C,C'\subset\Ovd$, are tangent at a point $P$ if and only if $C\cap C' = \{P\}$.
\end{lem}

\begin{proof}
    As tangent circles, $C\cap C'$ must not be empty, and cannot have more than 2 points, else their defining planes would intersect in a line that intersects $\Ovd$ in more than 2 points, an impossibility.
    
    But $T_PC = T_PC'$, which must be the intersection of their defining planes and lie within $T_P\Ovd$. Thus this line cannot intersect $\Ovd$ again, so $C$ and $C'$ cannot intersect again.
\end{proof}

\begin{lem}\label{bundletan}
    At any smooth $\Fq$-point $P$ of a curve $C_f\subset\Ovd$, there are $q$ circles tangent to $C_f$. These circles are tangent to each other at $P$ and their union contains every $\Fq$-point in the ovoid $\Ovd$. 
\end{lem}

\begin{proof}
    Because $P$ is smooth, we can construct the tangent line $L=T_PC_f$ as in Definition~\ref{tangline}. There are $q+1$ planes containing $L$, of which one is $T_P\Ovd$. Each of the remaining $q$ planes intersects $\Ovd$ in a circle.
    
    Each of the $q$ circles considered above is defined by a plane containing $L\subset T_P\Ovd$ and so $L$ is its tangent line at $P$. Thus all of these $q$ circles are tangent to each other at $P$. By Lemma~\ref{circtangency} they share no more intersection points. Each of these $q$ circles contains $q+1$ points, of which one is $P$ itself. Thus their union contains $q^2+1$ $\Fq$-points, which is all of $\Ovd(\Fq)$.
\end{proof}

\subsection{Involutions and Hyperbolic planes}\label{hyperbolic algebra} 

For this subsection, and for Section~\ref{sec:hyp} that relies on this one, we restrict our attention to even values of $q$.

Define an involution $*\colon\Pp^3\to\Pp^3$ by
\[
    (x:y:z:w)^* = (w:y:z:x).
\]
The fixed points $(x:y:z:w)$ of $*$ are those that satisfy $x=w$. But since $q$ is even, $x=w$ is the same as $x=-w$, so the fixed points of $*$ are precisely the points on the plane defined by the polynomial $b(x,y,z,w)=x+w$. To that end, fix the circle $B = C_b$ in $\Ovd$ as a `base' circle.

We will call a polynomial $f\in\Fq[x,y,z,w]$ \textbf{$xw$-symmetric} if $f(x,y,z,w) = f(w,y,z,x)$. The corresponding surface $S_f\subset\Pp^3$ is also called \textbf{$xw$-symmetric}. One can see that $xw$-symmetric surfaces are invariant under $*$, that is $P^*\in S_f$ if and only if $P\in S_f$.

Crucially, $\Ovd$ is $xw$-symmetric since $\Ovd$ is defined by the $xw$-symmetric polynomial $\varphi(x,y,z,w) = xw-y^2-yz-\gamma z^2$. For $xw$-symmetric functions $f$, we will also call the corresponding curves $C_f$ $xw$-symmetric. In particular, the circle $B$ is $xw$-symmetric.

In \cite{hyperbcrowe}, models for finite hyperbolic planes are constructed from inversive planes. The crucial point we will use from that paper is the following.
\begin{lem}[\cite{hyperbcrowe}]\label{crowe}
    Let $IG(2,q)$ be an inversive plane of even order $q$. If $P$ is a point not on the circle $C$, then the circles containing $P$ and tangent to $C$ are exactly the $q+1$ circles containing $P$ and some other point $P^*$. We call the pair $P,P^*$\textit{ hyperbolic conjugates}.
\end{lem}

In our case, $\Ovd(\Fq)$ is a model of an inversive plane, and $q$ is even. Any circle $C\subset\Ovd$ that is tangent to $B$ at a point $Q$ must be defined by a linear combination of the functions defining $B$ and $T_Q\Ovd$. Since both of those defining functions are $xw$-symmetric, so is the circle $C$. So if $C$ is a circle tangent to $B$ and containing a point $P\not\in B$, then $P^*\in C$ as well. That is, the $P^*$ defined in terms of the involution $*$ is the same as the $P^*$ in Lemma~\ref{crowe}.

A finite hyperbolic plane can thus be constructed using pairs $\{P,P^*\}$ for $P\in\Ovd(\Fq)\backslash B$ as points, and $xw$-symmetric circles (aside from $B$) as lines. The axioms and properties of finite hyperbolic planes that we will find most useful are summarized here. See e.g. \cite{hyperbcrowe} for more details.
\begin{itemize}
    \item Every pair of points in $(\Ovd\backslash B)/*$ determines a unique line, i.e. $xw$-symmetric circle containing the corresponding points in $\Ovd$.
    \item There are $\frac{q^2-q}{2}$ points in $(\Ovd\backslash B)/*$, and $q^2-1$ lines.
    \item There are $q/2$ points on each line, and $q+1$ lines through each point.
    \item Given a line $L$ and a point $P\not\in L$, there are $\frac{q}{2}+1$ lines containing $P$ and not intersecting $L$.
\end{itemize}

We thus refer $xw$-symmetric curves in $\Ovd$ as \textbf{hyperbolic curves}. Since every hyperbolic curve can be viewed also as an inversive curve, we can find a projective line in $\Pp^3$ tangent to the curve. To define the hyperbolic tangent line we must find a hyperbolic line, i.e. an $xw$-symmetric curcle, that is also tangent to that projective line when viewed as an inversive circle. In particular, the tangent line to a hyperbolic line should be the line itself.

Given a hyperbolic curve $C_f$ and a smooth point $P\in\Ovd\backslash B$, the projective line $L = T_PC_f = T_PS_f\cap T_P\Ovd$ from Definitions~\ref{tangline} is tangent to both $S_f$ and $\Ovd$ at $P$. The inversive circles through $P$ and tangent to $B$ are all $xw$-symmetric, and no pair of them are tangent to each other. Thus precisely one of them has the tangent line $L$.

\begin{defn}[Hyperbolic tangent line]\label{hyptangline}
    Let $C = C_f\subset\Ovd$ be a hyperbolic curve defined by an $xw$-symmetric function $f$, and $P\in C\backslash B$ a point. If $C$ is smooth at $P$ then the hyperbolic tangent line to $C$ at $P$ is the unique $xw$-symmetric circle $C_\ell$ such that $T_PS_\ell\cap T_P\Ovd = T_PS_f\cap T_P\Ovd$.
    
    If $C$ is not smooth at $P$, we define $T_PC = \Ovd$.
\end{defn}

\subsection{Transverse-freeness}

We begin with a very useful lemma for understanding tangencies to hypersurfaces.

\begin{lem}\label{cbez}
    The intersection set of a $k$-flat $\pi$ and a hypersurface $S_f$ of degree $d > 1$ defined over $\Fq$ is either $\pi$ itself or can be identified with a hypersurface of degree $d$ in $\Pp^k$.
\end{lem}

\begin{proof}
    The $k$-flat $\pi$ is defined as the intersection $S_{L_1}\cap\cdots\cap S_{L_{n-k}}$ for linearly independent linear forms $L_1,\ldots,L_{n-k}$. After rearranging the variables $x_0,\ldots,x_n$ if necessary, we can find linear forms $G_{k+1},\ldots,G_n$ so that
    \[
        P=(x_0:x_1:\cdots:x_n)\in\pi \qquad\Longleftrightarrow\qquad x_i = G_i(x_0,\ldots,x_k)\text{ for all }k+1\leq i\leq n.
    \]

    Now $(x_0:x_1:\cdots:x_n)\in \pi\cap S_f$ if $x_i=G_i(x_0,\ldots,x_k)$ for all $k+1\leq i\leq n$ and
    \[
        f(x_0,\ldots,x_k,G_{k+1}(x_0,\ldots,x_k),\ldots,G_n(x_0,\ldots,x_k)) = 0.
    \]
    This is either the zero polynomial (in which case $\pi\cap S_f=\pi$), or a degree $d$ homogeneous polynomial in $\Fq[x_0, \dots, x_k]$. In the latter case, $\pi\cap S_f$ is identified with the degree $d$ curve in $\Pp^k$ defined by the function $f(x_0,\ldots,x_k,G_{k+1}(x_0,\ldots,x_k),\ldots,G_n(x_0,\ldots,x_k))$.
\end{proof}

Using the identification of Lemma\ref{cbez}, we can identify $k$-flat tangencies with singularities of the intersection hypersurface. 
\begin{lem}\label{singtang}
    A $k$-flat $\pi\subset\Pp^n$ is tangent to the smooth hypersurface $S_f$ if and only if $\pi\cap S_f$ is identified, via Lemma~\ref{cbez}, with a singular hypersurface in $\Pp^k$.
\end{lem}

\begin{proof}
    The identification of Lemma~\ref{cbez} identifies the $k$-flat $\pi$ with $\Pp^k$. Working for now in $\pi\subset\Pp^n$, we see $\bar{S}=S_f\cap\pi$ has tangent space $T_P\bar{S} = T_PS_f\cap\pi$ at a point $P\in S_f\cap\pi$. Since $S_f$ is smooth, $T_Pf$ is a hyperplane. There are two cases now; either $\pi\subset T_PS_f$, or else $T_P\bar{S}$ is a $(k-1)$-flat.

    In the first case, viewing things again in $\Pp^k$, we see $T_P\bar{S}$ is all of $\Pp^k$, which means $\bar{S}$ is singular at $P$. In the second case, $T_P\bar{S}$ is the only $(k-1)$-flat within $\pi$ that is tangent to $S_f$ at $P$. After identifying with $\Pp^k$, there is exactly one hyperplane tangent to the hypersurface identified with $S_f\cap\pi$, so this hypersurface is smooth at $P$.

    Taken together, the hypersurface of $\Pp^k$ identified with $S_f\cap\pi$ is smooth if it is smooth at all points, if and only if $\pi$ is nowhere tangent to $S_f$.
\end{proof}

We will define transverse $k$-flats to be those that are not tangent, and using Lemma~\ref{singtang} we can phrase this definition in terms of the smoothness of the intersection.

\begin{defn}[Transverse $k$-flats]
    We say the $k$-flat $\pi$ is transverse to $S$ if $S\cap\pi$ corresponds, via Lemma~\ref{cbez}, to a smooth hypersurface in $\Pp^k$.
\end{defn}

\begin{rmk}\label{dim0smth}
    In the case of $k = 1$, i.e. $\pi=L$ is a line, the intersection $S\cap L$ is identified with a ``hypersurface" in $\Pp^1$. The defining function can be written as a homogeneous polynomial $g(x,y)$. If $S$ has degree $d$ then so does $g$, and we can factor $g(x,y)$ as a product of $d$ homogeneous linear factors. Then the singularity condition $g(x,y)=0$, $\nabla g(x,y)=0$ translates to a repeated factor. Thus we may say that $S\cap L$ is smooth if $S\cap L$ consists of exactly $d=\deg(g)$ distinct points.
\end{rmk}

One can see the combination of Lemma~\ref{singtang} and Remark~\ref{dim0smth} as saying that a line $L\subset\Pp^n$ is tangent to a degree $d$ hypersurface $S$ if $L\cap S$ does not consist of exactly $d$ points. From  B\'{e}zout's theorem, we know $d$ is the maximum finite number of intersections a degree $d$ hypersurface and a line can have. Now we may define the basic object of study for Sections~\ref{ndimproj} and \ref{sec:cones} below.

\begin{defn}\label{def:ktf}
    A projective hypersurface $S \subset \Pp^n$ defined over $\Fq$ is \ktf\, if no $k$-flat in $\Pp^n$ defined over $\Fq$ is transverse to $S$, i.e. if every $k$-flat defined over $\Fq$ is tangent to $S$.

    If $S$ is \ktf\ for $k=1$, we say $S$ is \ltf. If $S$ is \ktf\ for $k=2$, we say $S$ is \ptf.
\end{defn}

In the ovoid $\Ovd$, circles are defined in terms of planes in $\Pp^3$, and the tangent line to a circle at a point $P$ lies both in its corresponding plane and in $T_P\Ovd$. Thus we may view a tangency between an inversive curve and a circle within the plane that defines a circle. Likewise we can define transverse circles to inversive curves.

\begin{defn}[Transverse circles]
    The circle $C_\ell\subset\Ovd$ is transverse to the curve $C_f\subset\Ovd$ if $C_\ell$ is not tangent to $C_f$ at any $P\in C_\ell\cap C_S$.
\end{defn}

Circles in $\Ovd$ are defined in terms of planes in $\Pp^3$, and from Lemma~\ref{cbez} we can identify the intersection of a plane with any other surface as a curve in $\Pp^2$. One can phrase the above notion of transversality in terms of the number of intersections of these curves, analogous to the observation about using B\'{e}zout's theorem for intersections of hypersurfaces and lines. However we found this connection to be not extremely useful for our purposes, so we will not pursue it.

For inversive curves, our main object of study are those that are tangent to every circle.

\begin{defn}[Circle-transverse-free]
    An inversive curve $C\subset\Ovd$ defined over $\Fq$ is \ctf\ if it is tangent to every $\Fq$-circle in $\Ovd$.
\end{defn}

And a hyperbolic curve is simply an $xw$-symmetric curve in $\Ovd$. The lines in the hyperbolic plane are simply the $xw$-symmetric circles in $\Ovd$ aside from $B$. Using the same notions of tangency and transversality as above, we may define hyperbolic \ltf\ curves as follows.

\begin{defn}
    A hyperbolic curve $C$ defined over $\Fq$ is \ltf\ if it is tangent to every hyperbolic line, i.e. every $xw$-symmetric $\Fq$-circle in $\Ovd$.
\end{defn}

\section{Transverse-free hypersurfaces in projective space}\label{ndimproj}

Our first results are focused on smooth, \spf\, \ktf\, hypersurfaces in $\Pp^n$ of low degree. By \spf\, we mean containing $\Pp^n(\Fq)$. 

\begin{defn}[Space-filling]\label{spf}
    A hypersurface $S_f \subset \Pp^n$ is \spf\ if $\Pp^n(\Fq) \subseteq S_f$, i.e. $f(P)=0$ for all $P\in\Pp^n(\Fq)$.
\end{defn}

For the sake of convenience, we denote

\[
    \SSFS_n(d) = \{\text{smooth, \spf\ hypersurfaces in }\Pp^n\text{ of degree }d\}.
\]

\begin{rmk}\label{minimal degree}
    It is well known that \spf\ hypersurfaces exist only in degrees $d\geq q+1$. And it is shown in Proposition 14 of \cite{hk-bound} that if the dimension $n$ is even, $\SSFS_n(q+1)$ is empty.

    The set $\SSFS_2(q+2)$ is shown to be non-empty in \cite{hk-singular}, and \cite{asgarli-lines} claims that $\SSFS_3(q+2)$ is non empty when $q\leq31$ is prime.. For odd $n$, \cite{space-filling} gives examples of hypersurfaces in $\SSFS_n(q+1)$.
\end{rmk}

Inspired by the examples of smooth \ltf\ curves in $\Pp^2$ exhibited in \cite{transverse-free}, in this section, we analyze low-degree smooth \spf\ hypersurfaces in order to produce some examples of \ktf\ hypersurfaces for  $k=1$ and $k$ even. We also obtain some results showing the relationship (or lack thereof) between \ltf\ and \ptf\ hypersurfaces of low degree.

\subsection{Degree q+1 hypersurfaces}

We first consider the minimal degree $d = q+1$, with $n$ odd. From Remark~\ref{minimal degree} we can immediately conclude:

\begin{thm}{\label{qp1AlwaysPTF}}
    If $n$ is odd and $k\geq2$ is even, then every $S \in \SSFS_n(q + 1)$ is \ktf.
\end{thm}
\begin{proof}
    By Definition~\ref{def:ktf}, $S$ is \ktf\ iff for every $k$-flat $\pi$, $S \cap \pi$ has a singularity or $S \cap \pi = \pi$. From Lemma \ref{cbez}, every variety $S \cap \pi$ is either $\pi$, which implies $\pi$ is tangent to $S$, or is identified with a hypersurface of degree $q+1$ in $\Pp^k$. Because $S$ is \spf, $S \cap \pi$ is \spf. Proposition 14 of \cite{hk-bound} says that no smooth \spf\ hypersurfaces in $\Pp^k$ of degree $q+1$ exist because $k$ is even. Hence in this case too, it follows that $\pi$ is tangent to $S$.
\end{proof}

We further comment on the \ltf-ness of \spf\ hypersurfaces of minimum degree.

\begin{lem}\label{(Ammel Gnizama)}
    If $L$ is an $\Fq$-line tangent to some $S \in \SSFS_n(q+1)$, then $L$ is completely contained within the hypersurface $S$, i.e. $L \subset S$.
\end{lem}

\begin{proof}
     Given a line $L$ tangent to $S$, B\'{e}zout's Theorem tells us that if $L$ and $S$ intersect at finitely many points, then there are exactly $\deg{S}=q+1$ intersection points counted with multiplicity. Because $S$ is space filling, it must intersect each of the $q+1$ $\Fq$-points on $L$, and since $L$ is tangent to $S$ at one of these points, one of those intersections has multiplicity $\geq2$.

     So if $L$ is not contained in $S$ then it must have at least $q+2$ intersections with $S$, counted with multiplicity. But this is greater than the $q+1$ intersections that B\'{e}zout's Theorem decrees, a contradiction. Thus $L$ must be contained in $S$.
\end{proof}

Using Lemma~\ref{(Ammel Gnizama)}, we arrive at the following result, which completes our tour of hypersurfaces of degree $q+1$.

\begin{thm}{\label{qp1NeverLTF}}
   No hypersurface $S \in \SSFS_n(q+1)$ is \ltf. 
\end{thm}

\begin{proof}
    Suppose for the sake of contradiction that a hypersurface $S\in\SSFS_n(q+1)$ is \ltf. $S$ is space filling, so each of the $\pi_q(0,n)=\frac{q^{n+1}-1}{q-1}$ points of $\Pp^n(\Fq)$ is contained in $S$, and since $S$ is smooth it has a tangent hyperplane at each of these points. According to Lemma~\ref{duallem}, the number of lines within a hyperplane which pass through a given point is $\pi_q(n-3,n-2)=\frac{q^{n-1}-1}{q-1}$.
    
    Each of these lines is, by our \ltf\ assumption, tangent to $S$ and thus, by Lemma~\ref{(Ammel Gnizama)}, contained within $S$. Consider the set pairs $(P,L)$ where $L$ is an $\Fq$-line tangent to $S$ at the $\Fq$-point $P$. From Lemma~\ref{(Ammel Gnizama)}, each $P\in S(\Fq)$ contributes a number of pairs to this set equal to $\pi_q(n-3,n-2)$, so we compute
    \[
        \#\{(P,L)\mid S\text{ is tangent to }L\text{ at }P\} = \pi_q(0,n) \pi_q(n-3,n-2) = \frac{(q^{n+1}-1)(q^{n-1}-1)}{(q-1)^2}.
    \]
    
    Note that each of the $\frac{q^n-1}{q-1}$ lines defined over $\Fq$ is included in this count since every line $L$ is tangent to $S$. And from Lemma~\ref{(Ammel Gnizama)} we know that each line tangent to the hypersurface $S$ must be completely contained within $S$, and so each line has $q+1$ $\Fq$-points of tangency. Thus each line appears $q+1$ times in the count above. In other words, the number of lines accounted for in the expression is
    \[
        \frac{(q^{n+1}-1)(q^{n-1}-1)}{(q^2-1)(q-1)}.
    \]
    But this is less than $\pi_q(1,n) = \frac{(q^{n+1}-1)(q^{n}-1)}{(q^2-1)(q-1)}$, the total number lines in $\Pp^n(\Fq)$. This contradicts our assumption that $S$ was \ltf.
\end{proof}

The two theorems so far show an interesting dichotomy between \ktf\-ness for different values of $k$. While we do not address the case $k>1$ odd, we have seen that every $S\in\SSFS_n(q+1)$ is \ktf\ for all $k\geq 2$ even but not for $k=1$.

\subsection{Degree q+2 hypersurfaces}

Next we turn our attention to the next smallest case of smooth \spf\ hypersurfaces of degree $d=q+2$. Again we can say something about intersections with lines.

\begin{prop}\label{clsfying-lines}
    For $S \in \SSFS_n(q + 2)$, every $\Fq$-line is either tangent to $S$ at exactly one $\Fq$-point or is completely contained in $S$ and therefore tangent to $S$ at $q+1$ $\Fq$-points.
\end{prop}

\begin{proof}
    By B\'{e}zout's Theorem, either the line is included in $S$ and thus intersects $S$ at infinitely many points, or it intersects $S$ at $q + 2$ points counted with multiplicity. We need only explore the latter case. Since $S$ is \spf, we have that every $\Fq$-point on the line has intersection multiplicity at least one.
    
    Hence, we need only eliminate the case where all intersection multiplicities between $S$ and some line $L$ are exactly one. Under this assumption, $L$ must intersect $S$ at exactly one non-$\Fq$-point, $P=(p_0:p_1:\cdots,p_n)$. But since both the line and $S$ are defined over $\Fq$, the Frobenius conjugate, $P^{\sigma}=(p_0^q:p_1^q:\cdots,p_n^q)$, is also an intersection point. But then $P$ was not the only non-$\Fq$-point in $S\cap L$, a contradiction. Thus one $\Fq$-point $Q\in S\cap L$ must have multiplicity 2, causing $L$ to be tangent to $S$ at $Q$. In other words, every line not contained in $S$ is tangent to $S$ at exactly one $\Fq$-point.
\end{proof}

An immediate consequence of the above Proposition is the following.

\begin{cor}{\label{qp2AlwaysLTF}}
    Every $S \in \SSFS_n(q + 2)$ is \ltf.
\end{cor}

Above we saw that the notions of \ltf\ and \ptf\ do not coincide for smooth \spf\ hypersurfaces of degree $q+1$. In degree $q+2$ we have just seen that every smooth \spf\ hypersurface is \ltf. We will proceed to show that in dimension $n=3$, no $S\in SSFS_3(q+2)$ is \ptf when $q$ is odd. With this, we have two classes of surfaces, $\SSFS_3(q+1)$ and $\SSFS_3(q+2)$ for $q$ odd, one of which is always \ltf\ and never \ptf, and the other is always \ptf\ and never \ltf.

The tangent lines to $S$ give us some information about the tangent planes, and in particular we can leverage Proposition \ref{clsfying-lines} to give the following result which will prove essential to our goal.

\begin{prop}\label{inclines}
    Exactly $\frac{(q^{n+1}-1)(q^{n-2}-1)}{(q^2-1)(q-1)}$ $\Fq$-lines are included in every $S \in \SSFS_n(q + 2)$.
\end{prop}

In particular, we will use this result for $n=3$, which gives $q^2+1$ included lines.

\begin{proof}
    We proceed by counting the number of pairs $(P,L)$ where the line $L$ is tangent to $S$ at the $\Fq$-point $P$. Let $\ell_0$ be the number of lines contained in $S$ and let $\ell_1$ be the number of lines tangent to $S$ at exactly one $\Fq$-point. We know the total number of $\Fq$-lines in $\Pp^n(\Fq)$ is $\pi_q(1,n)$, so Proposition \ref{clsfying-lines} gives us the partition
    
    \[
        \ell_0 + \ell_1 = \pi_q(1,n).
    \]
    
    Since $S$ is smooth, \spf, and defined over $\Fq$, for every $P\in\Pp^n(\Fq)$ there is exactly one $\Fq$-hyperplane $\pi$ tangent to $S$ at $P$. Just as in the proof of Theorem~\ref{qp1NeverLTF}, we have from Lemma \ref{duallem} the number of lines in $\pi$ that pass through $P$ as $\pi_q(n-3,n-2)$. All of these $\pi_q(n-3,n-2)$ lines are tangent to $S$.
    
    For all $P\in\Pp^n(\Fq)$, there are $\pi_q(n-3,n-2)$ lines tangent to $S$ at $P$. so there are $\pi_q(0,n)\pi_q(n-3,n-2)$ point-line tangency pairs. Alternatively, each line contained within the hypersurface contributes $q+1$ points of tangency, while every other line contributes a unique point of tangency. We can equate these two ways of counting to see
    \[
        (q + 1)\ell_0 + \ell_1 = \pi_q(0,n)\pi_q(n-3,n-2).
    \]
    Solving the resultant system of equations shows that
    \[
        \ell_0 = \frac{(q^{n+1}-1)(q^{n-2}-1)}{(q^2-1)(q-1)}
    \]
    as claimed.
\end{proof}

There is one more ingredient we'll need before we can analyze the property of being \ptf.

\begin{lem}\label{hklem}
    For $S \in \SSFS_n(q+2)$, if $\pi$ is an $\Fq$-plane tangent to $S$, then $\pi$ has at least one $\Fq$-point of tangency with $S$.
\end{lem}

\begin{proof}
    We are done if $\pi \subset S$ so assume otherwise. By Lemma \ref{cbez}, $S\cap\pi$ is identified with a curve in $\Pp^2$ of degree $q + 2$. Since $S$ is \spf, $S\cap\pi$ is necessarily plane-filling. Moreover by Lemma~\ref{singtang}, we know the curve $S\cap\pi$ is singular. By Claims 1 and 2 in the proof of Theorem 3.2 in \cite{hk-singular}, a plane-filling curve of degree $q + 2$ is singular if and only if it is singular at an $\Fq$-point. This is precisely our desired result.
\end{proof}

In dimension $n=3$ we have the following corollary.

\begin{cor}\label{ptplcor}
    If $S \in \SSFS_3(q + 2)$ is \ptf\, then there exists a perfect matching between $\Fq$-planes and $\Fq$-points induced by tangency to $S$.
\end{cor}

\begin{proof}
    Note that $\pi_q(0,3)=\pi_q(2,3)$, i.e. the number of $\Fq$-points and planes in $\Pp^3$ are the same. Moreover no two planes can be tangent to $S$ at the same point because $S$ is smooth. The perfect matching therefore follows from Lemma \ref{hklem}.
\end{proof}


Our strategy for analyzing smooth, \spf\ \ptf\ hypersurfaces will be to show that a perfect matching is never attainable. Namely we will show that for every $S \in \SSFS_3(q+2)$, there exists an $\Fq$-plane that is tangent to $S$ at more than one $\Fq$-point.

\begin{thm}{\label{qp2NeverPTF}}
    When $q$ odd, there do not exist \ptf\ hypersurfaces $S_f \in \SSFS_3(q + 2)$.
\end{thm}

\begin{proof}
    Consider the map $\nabla{f}: S \to (\Pp^3)^*$ which maps $P\in S$ to the point $\nabla f(P)\in(\Pp^3)^*$ that is dual to the tangent plane of $S$ at $P$. This map is well-defined since $f$ is smooth. Our goal now is to prove $\nabla f$ is not bijective on $\Fq$-points. Then, using Corollary~\ref{ptplcor} we will show $S=S_f$ is not \ptf.
        
    Following \cite{hk-singular}, the set of polynomials in $\Fq[x_0,x_1,x_2,x_3]$ that vanish on all of $\Pp^3(\Fq)$ is generated by $\{X_i^qX_j - X_iX_j^q: 0 \leq i < j \leq 3\}$. Therefore the polynomial $f$ defining $S$ can be written
    \[
        f(x_0, x_1, x_2, x_3) = \sum_{0 \leq i < j \leq 3} L_{i,j}(x_i^qx_j - x_ix_j^q) = 0
    \]
    for some linear polynomials $L_{i,j}$.
        
    For the purposes of symmetry we'll also define $L_{j,i} = L_{i,j}$ for $i < j$. Let $[i,j]$ be $1$ when $i > j$ and $-1$ when $i < j$. Then at $\Fq$-points we know each $x_i^qx_j-x_ix_j^q = 0$ and since $q=0$ in $\Fq$, $\frac{\partial}{\partial{x_i}}(x_i^qx_j-x_ix_j^q)=x_j^q$. Hence restricted to $\Fq$-points,
    \[
        \frac{\partial{f}}{\partial{x_i}} = \sum_{j\neq i}[j,i]L_{i,j}x_j.
    \]
    As such the restriction of $\nabla{f}$ to $\Pp^3(\Fq)$ can be represented as a map of degree $2$.

    Now consider a line $L$ contained in $S$, which exists by Proposition~\ref{inclines}. At any $\Fq$-point $P\in L$, the tangent plane $\pi$ necessarily includes $L$. So we can say $L\subset\pi$, or in other words, $\pi^*\in L^*$. As such, the restriction of $\nabla f$ to the $\Fq$-points of $L$ corresponds to a map $\nabla f\vert_{L(\Fq)}:\Pp^1(\Fq)\to(\Pp^1(\Fq))^*$. And just as above, this restriction can be represented as a map of degree $2$.
    
    We will now proceed to show that this restriction is not bijective, which will imply the original map $\nabla f:\Pp^3(\Fq)\to(\Pp^3(\Fq))^*$ is not bijective.

    Algebraically we know that the degree 2 map $\nabla f\vert_{L(\Fq)}$ can be written as
    \[
        (x:y)\mapsto (P(x, y): Q(x, y))
    \]
    for $P, Q$ homogeneous polynomials of degree 2. Thus a preimage under $\nabla{f}$ can be computed as the solution to a degree 2 equation, and generically there are either 0 or 2 solutions over $\Fq$. More precisely, for fixed $(u:v)\in\Pp^1(\Fq)$, $\nabla{f}\vert_{L(\Fq)}(x:y) = (u:v)$ if and only if
    \[
        (vP - uQ)(x, y) = 0.
    \]
    
    Suppose $\nabla f\vert_{L(\Fq)}$ is bijective. Then there exists $(u_0: v_0) \in \Pp^1(\Fq)$ such that $(0:1)$ is the sole zero of $(v_0P - u_0Q)$, and therefore $(v_0P - u_0Q) = k_0x^2$ for some nonzero $k_0 \in \Fq$. Similarly, we can find $(u_1: v_1) \in \Pp^1(\Fq)$ such that $(v_1P - u_1Q) = k_1y^2$ for some nonzero $k_1 \in \Fq$. Therefore, $\nabla{f}\vert_{L(\Fq)}(x:y) = (u_0 + u_1: v_0 + v_1) \in \Pp^1(\Fq)$ if and only if
    \[
        k_0x^2 + k_1y^2 = 0.
    \]
    The solutions to this equation always come in non-degenerate pairs when $1 \neq -1$, which is true for $q$ odd. We thus conclude that $\nabla f\vert_{L(\Fq)}$ is not bijective.
    
    As a result, $\nabla f$ does not induce a perfect correspondence between $\Fq$-points and $\Fq$-planes of $\Pp^3$ and so by Corollary~\ref{ptplcor}, $S_f$ is not \ptf.
\end{proof}

\section{Examples of transverse-free hypersurfaces: cones}\label{sec:cones}

In this section, we'll explore a construction of \ktf\ hypersurfaces in high-dimensional projective spaces.

\begin{defn}{(Cones)}\label{def:cone}
    Given a hypersurface $S_f \subseteq \Pp^n$, define the \textit{cone over $S$ in dimension $n+k$} to be the hypersurface $S_F \subseteq \Pp^{n+k}$, where
    \[
        F(x_0,x_1,\ldots,x_{n+k}) = f(x_0,x_1,\ldots,x_n).
    \]
\end{defn}

We call these constructed hypersurfaces \emph{cones} based on another way to construct them. Given any point $(x_0:x_1:\cdots:x_n)\in S_f$ and any point $(x_{n+1}:\cdots:x_{n+k})\in\Pp^{k-1}$, we can manipulate their representatives in $\overline{\Fq}^{n+1}$ and $\overline{\Fq}^k$ to construct the line
\[
    \{(ax_0:ax_1:\cdots:ax_n:bx_{n+1}:\cdots:bx_{n+k})\mid (a:b)\in\Pp^1\}.
\]

From the homogeneity of $f$ and $F$, every point on this line is contained in the hypersurface $S_F$. And for fixed $P\in S$ and $Q$ ranging over $\Pp^{k-1}$, all these lines pass through the same point $P$, just as the lines of a cone pass through the vertex. In the special case $k=1$, $\Pp^0$ is a point and the hypersurface $S_F$ looks like a simple cone over $S$.

It is also worth noting the degree to which these hypersurfaces are singular. Even if we start with a smooth hypersurface $S$ and form its cone in one more dimension, the resulting hypersurface will be singular at the point $(0:0:\ldots,0:1)$. In general we can classify the singularities of $S_F$ in terms of the singularities of $S$ and the parameter $k$.

\begin{lem}\label{sing-cone}
    Let $S\subset\Pp^n$ be a degre $d\geq 2$ hypersurface with a set $\Sigma\subset S$ of singular points. The cone $S_k\subset\Pp^{n+k}$ is singular at $(x_0:x_1:\cdots:x_{n+k})$ if and only if $(x_0:x_1:\cdots:x_n)\in\Sigma$ or $x_0=x_1=\cdots=x_n=0$.
\end{lem}
\begin{rmk}
     The latter case corresponds to the space $\Pp^{k-1}\subset\Pp^{n+k}$. In analogy with the case $k=1$ described above, we think of this subspace as the vertex of the cone, which is expected to be singular.
\end{rmk}

\begin{proof}
    First note $\nabla F(x_0,\ldots,x_{n+k})$ is simply $\nabla f(x_0,\ldots,x_n)$ appended with $k$ 0's. Also, $F(x_0,\ldots,x_{n+k})=0$ precisely when $f(x_0,\ldots,x_n)=0$, which happens either when $(x_0:\cdots:x_n)\in S$ or $x_0=\cdots=x_n=0$.

    In the first case, where $(x_0:\cdots:x_n)\in S$, the point $(x_0:\cdots:x_{n+k})$ is a singularity for $S_F$ if $\nabla F(x_0,\ldots,x_{n+k})=0$, which happens precisely if $\nabla f(x_0,\ldots,x_n)=0$, i.e. $(x_0:\cdots:x_n)$ is a singularity for $S_f$. In the second case, where $x_0=\cdots=x_n=0$, we are at the vertex described above, where $\nabla f(0,\ldots,0) = 0$ as $\nabla f$ is homogeneous of degree $d\geq 1$.
\end{proof}

Now we can see that our construction of cones preserves the property of being \ktf.

\begin{thm}
    Let $S_f\subset\Pp^n$ be a \ktf\ hypersurface. For any $m\geq 1$, the cone $S_F\subset\Pp^{n+m}$ from Definition~\ref{def:cone} is also \ktf.
\end{thm}

\begin{proof}
    To prove this theorem, we will show that the cone in one dimension higher is also \ktf. The result for cones in higher dimensions then follows inductively. To that end, let $F(x_0,\ldots,x_{n+1}) = f(x_0,\ldots,x_n)$.

    Based on the singularities discussed in Lemma~\ref{sing-cone}, we can say that any $k$-flat passing through the vertex $(0:\cdots:0:1)$ of the cone $S_F$ is not transverse to $S_F$ since the cross section also has a singularity. at that point. A $k$-flat in $\Pp^{n+1}$ is defined by $n+1-k$ linear functions in the variables $x_0,\ldots,x_{n+1}$. The $k$-flat contains the point $(0:\cdots:0:1)$ if each of those linear functions has the form
    \[
        a_0x_0+a_1x_1+\cdots+a_nx_n.
    \]

    A $k$-flat in $\Pp^{n+1}$ not containing $(0:\cdots:0:1)$ is also defined by $n+1-k$ linear functions, at least one of which can be written in the form
    \[
        a_0x_0+a_1x_1+\cdots+a_nx_n - x_{n+1}.
    \]
    So for a given $k$-flat $\pi$, we can write $\pi = \pi_0\cap\pi_1\cap\cdots\cap\pi_{n-k}$ as the intersection of hyperplanes, where $\pi_0$ has the form above.

    Now, just as in Lemma~\ref{cbez}, the intersection of $S_F$ and $\pi_0$ can be identified with a hypersurface in $\Pp^n$, and the identification simply replaces $x_{n+1}$ in each function by the expression $a_0x_0+\cdots+a_nx_n$. Since $F$ has no dependence on $x_{n+1}$, the intersection hypersurface is defined by $f$, namely it is $S$. Identifying the intersections $\pi_1\cap\pi_0,\ldots,\pi_{n-k}\cap\pi_0$ with hyperplanes $\pi'_1,\ldots,\pi'_{n-k}\subset\Pp^n$ we see they intersect in a $k$-flat $\pi'=\pi'_1\cap\cdots\pi'_{n-k}$. Now we can see the intersection $S_F\cap\pi$ is identified via Lemma~\ref{cbez} with $S\cap\pi'$. Since $S$ is \ktf, this intersection is singular, so the original $k$-flat $\pi$ is tangent to $S_F$.

    Since $S_F$ is tangent to all $k$-flats that contain the vertex $(0:\cdots:0:1)$, as well as all $k$-flats that don't, it is \ktf.
\end{proof}

\section{Curves in the inversive plane}\label{sec:ctf}

In this section we investigate low-degree, \ctf\ inversive curves, with the goal of giving bounds on degrees in which \ctf\ curves are guaranteed to exist.


In analog to the \spf\ hypersurfaces of Section~\ref{ndimproj}, one can search for \ctf\ curves that contain $\Ovd(\Fq)$. If $\Ovd(\Fq)\subset S_f$ and $C_f$ is a smooth curve, then an application of B\'{e}zout's theorem shows the degree of $f$ must be at least $\lceil \frac{q+1}{2} \rceil$. Unfortunately we found equality cannot hold, i.e. no \ctf\ curves of these degrees containing $\Ovd(\Fq)$ exist.

Nevertheless, we may apply some linear algebraic tools, together with combinatorial results about finite projective and inversive geometries, to construct \ctf\ curves in relatively small degrees. The next susbsection develops the linear algebraic tools we will use both in this section, and in Section~\ref{sec:hyp} below for hyperbolic curves.

\subsection{Tangency and singularity conditions for curves}\label{curve conditions}

We can construct \ctf\ inversive curves and \ltf\ hyperbolic curves similarly to the constructions of transverse free curves in $\Pp^2$ exhibited in \cite{freidin-asgarli}. The strategy is to specify a number of tangency and singularity conditions for curves to satisfy which will ensure the curve are transverse-free. These tangency and singularity conditions will impose a system of constraints on the coefficients of a polynomial defining such a curve.

An inversive curve is defined as $C_f=S_f\cap\Ovd$ for some homogeneous polynomial $f$, and a hyperbolic curve is defined by a $xw$-symmetric homogeneous polynomial. So consider the $\Fq$-vector space $\mathcal{F}(d)$ of degree $d$ homogeneous polynomials in $n+1$ variables, and the subspace $\mathcal{H}(d)$ of $xw$-symmetric polynomials. We must first know their dimensions.

\begin{lem}\label{dims}
    \[
        \dim(\mathcal{F}(d)) = \binom{d+3}{3}
    \]
    and
    \[
        \dim(\mathcal{H}(d))=\begin{cases}
            \frac{1}{2}(\binom{d+3}{3}+\frac{1}{4}(d+2)^2 & \text{$d$ even} \\ \frac{1}{2}(\binom{d+3}{3}+\frac{1}{4}(d+1)(d+3)) & \text{$d$ odd}.
        \end{cases}
    \]
\end{lem}
\begin{proof}
    The space $\mathcal{F}(d)$ is spanned by homogeneous monomials, that is polynomials of the form $x^Ay^Bz^Cw^D$ with $A+B+C+D=d$. It is well known that there are $\binom{d+3}{3}$ choices of $A,B,C,D\geq 0$ satisfying this equation.
    
    The space $\mathcal{H}(d)$ is spanned by $xw$-symmetric monomials of degree $d$, alongside polynomials of the form $m(x,y,z,w)+m(w,y,z,x)$ for non-$xw$-symmetric monomials $m$. The $xw$-symmetric monomials have the form $x^ay^bz^cw^a$ where $2a+b+c=d$. For each fixed $a$ between 0 and $d/2$, there are $d-2a+1$ choices of exponents $b$ and $c$ to complete a monomial of degree $d$. Thus the total number of $xw$-symmetric monomials of degree $d$ is
    \[
        \sum_{a=0}^{\lfloor d/2\rfloor}(d-2a+1) = (d+1-\lfloor d/2\rfloor)\left(\lfloor d/2\rfloor+1\right) = \begin{cases} (d+2)^2/4 & d\text{ even}\\(d+1)(d+3)/4 & d\text{ odd}.\end{cases}
    \]

    Out of the original $\binom{d+3}{3}$ monomials of degree $d$, the remaining ones are not $xw$-symmetric. Pairing them up as described above, gives the rest of the basis for $\mathcal{H}(d)$. If $s$ is the number of $xw$-symmetric polynomials computed above, the dimension of $\mathcal{H}(d)$ is then
    \[
        s+\frac{1}{2}\left(\binom{d+3}{3}-s\right) = \frac{1}{2}\left(\binom{d+3}{3}+s\right).
    \]
\end{proof}

Now suppose $f$ is a homogeneous polynomial, or an $xw$-symmetric polynomial. Then it can be written as a linear combination of the basis for $\mathcal{F}(d)$, or the basis for $\mathcal{H}(d)$, described above. That is, we can write $f=a_1f_1+\cdots+a_Nf_N$ where the polynomials $f_i$ form a basis for the relevant space.

Now we can write conditions for the coefficients $a_1,\ldots,a_N$ satisfied by a polynomial $f$ that defines a curve with various conditions. We begin with the larger space $\mathcal{F}(d)$.

\begin{lem}\label{contains-pt}
    Fix $P\in\Fq^4$. The space of polynomials $f\in \mathcal{F}(d)$ with $f(P)=0$ has codimension 1 in $\mathcal{F}(d)$, i.e. dimension $\dim(\mathcal{F}(d))-1$.
\end{lem}
\begin{proof}
    The condition $f(P) = 0$ can be rewritten
    \[
        a_1f_1(P) + \cdots + a_Nf_N(P) = 0.
    \]
    Each $f_i(P)$ is a constant, so this imposes one linear relation in the coefficients $a_1,\ldots,a_N$. Thus the space of polynomials satisfying this condition has codimension 1, i.e. dimension $dim(\mathcal{F}(d))-1$.
\end{proof}

\begin{lem}\label{tang-cond}
    Fix $P\in\Ovd(\Fq)$ and a line $L\subset T_P\Ovd$ containing $P$. The space of polynomials $f\in\mathcal{F}(d)$ defining curves $C_f$ tangent to $L$ at $P$ has codimension at most 2 in $\mathcal{F}(d)$.
\end{lem}
\begin{proof}
    A curve $C_f$ is tangent to $L$ at $P$, as a curve in $\Pp^3$, if $P\in C_f$ and $L$ is tangent to $S_f$. The condition $P\in C_f$ just says $f(P) = 0$, and by Lemma~\ref{contains-pt} this imposes 1 linear relation on the coefficients $a_1,\ldots,a_N$.

    If $P=(p_0:p_1:\cdots:p_n)$, and $Q=(p_0+v_0:p_1+v:\cdots:p_n+v_n)$ is another $\Fq$-point on $L$, then the condition $L\subset T_PS$ is encoded by the equation
    \[
        v_0\frac{\partial f}{\partial x_0}(P) + \cdots + v_n\frac{\partial f}{\partial x_n}(P) = 0.
    \]
    Each $\frac{\partial f}{\partial x_i}(P)$ is a linear combination of the coefficients $a_1,\ldots,a_N$, so this tangency condition imposes one more relation on the coefficients.

    These two linear constrainst may not be linearly independent if $d$ is too small, but the polynomials satisfying both of them lie in a subspace of dimension at least $\dim(\mathcal{F}(d))-2$, i.e. codimension at most 2.
\end{proof}

Note that $C_f$ is tangent to the line $L$ if and only if $C_f$ is tangent to any (and thus all) of the $q$ circles of $\Ovd$ that are themselves tangent to $L$.

\begin{lem}\label{sing-cond}
    Fix $P\in\Ovd(\Fq)$. The space of polynomials $f\in\mathcal{F}(d)$ defining curves $C_f$ singular at $P$ has codimension at most 3 in $\mathcal{F}(d)$.
\end{lem}
\begin{proof}
    The curve $C_f$ is singular at $P\in\Ovd(\Fq)$ if $T_P\Ovd$ is tangent to the surface $S_f$ at $P$. This is captured by saying $P\in S_f$ and $\nabla f(P)$ is a scalar multiple of $\nabla\varphi(P)$ for the defining polynomial $\varphi$ of $\Ovd$. Just as above, Lemma~\ref{contains-pt} provides one linear relation among the coefficients $a_1,\ldots,a_N$ that guarantees $P\in S_f$.

    Picking 2 independent directions in $T_P\Ovd$, we can say that $C_f$ is singular at $P$ if $P\in S$ and $\nabla f(P)$ is perpendicular to both selected directions in $T_P\Ovd$. This induces 2 more linear constraints on the coefficients $a_1,\ldots,a_N$, for a total of 3 constraints which may or may not be linearly independent. Thus these polynomials lie in a subspace of codimension at most 3.
\end{proof}

Of course, the situation where $C_f$ is tangent to a line $L$ at a point $P$, as described in Lemma~\ref{tang-cond}, includes the possibility that $C_f$ is singular at $P$ as described in Lemma~\ref{sing-cond}.

\begin{rmk}\label{deg-r}
    If any of the points $P\in\Ovd$ above are non-$\Fq$-points, say the coordinates of $P$ are all in the degree $r$ extension $\mathbb{F}_{q^r}$. This extension can be seen as an $r$-dimensional vector space over $\Fq$, so each linear condition described above becomes $r$ linear conditions at a degree $r$ point.
\end{rmk}

\begin{prop}\label{codims}
    Fix distinct points $P_1,\ldots,P_a,Q_1,\ldots,Q_b$ in $\Pp^3$, and a line $L_i$ tangent to $\Ovd$ at $P_i$ for each $i=1,\ldots,a$. Suppose each $P_i$ has degree $r_i$ and each $Q_i$ has degree $\rho_i$. The set of polynomials $f\in\mathcal{F}(d)$ defining curves $C_f$ that are tangent to each $L_i$ at $P_i$ and singular at each $Q_i$ has codimension at most
    \[
        \sum_{i=1}^a2r_i + \sum_{i=1}^b3\rho_i
    \]
    in $\mathcal{F}(d)$.
\end{prop}
\begin{proof}
    From Lemma~\ref{tang-cond} and Remark~\ref{deg-r}, the condition that a curve $C_f$ is tangent to $L_i$ at $P_i$ determines a subspace of codimension at most $2r_i$ in $\mathcal{F}(d)$. And from Lemma~\ref{sing-cond} and Remark~\ref{deg-r}, the condition that a curve $C_f$ is singular at $Q_i$ determines a subspace of  codimension at most $3\rho_i$.

    Combining all the tangency and singularity conditions amounts to intersecting all the subspaces described above, and the codimension of an intersection is at most the sum of the individual codimensions. All the subspaces are linear, since the zero polynomial satisfies all the homogeneous constraints listed, so the intersection is non-empty.
\end{proof}

In particular, if $\sum 2r_i + \sum 3\rho_i < N$, there are definitely non-zero polynomials defining curves with all the listed conditions. Of course, for any polynomials $f$ and $g$ such that $g-f$ is divisible by the quadratic polynomial $\varphi$ defining $\Ovd$, the polynomials $f$ and $g$ define the same inversive curve. To eliminate this over-counting, we consider the space of degree $d$ polynomials that are multiples of the quadratic $\varphi$. These are precisely the surfaces that contain $\Ovd$.

\begin{lem}\label{ovoidspace}
    The dimension of degree $d$ polynomials which define surfaces containing $\Ovd$ is $\binom{d+1}{3}$. 
\end{lem}

\begin{proof}
    If a hypersurface $S_f$ contains $\Ovd = S_{\varphi}$, then $f$ may be factored as $f = \varphi \cdot g$. The factor $g$ must have degree $d - 2$ as $\varphi$ has degree 2, and the space of polynomials with degree $d-2$ is given by $\binom{d + 1}{3}$. Thus the space of polynomials of degree $d$ defining surfaces containing $\Ovd$ is $\binom{d + 1}{3}$. 
\end{proof}

\begin{rmk}\label{same for hyp}
    All of the results of this section, from Lemma~\ref{contains-pt} through Lemma~\ref{ovoidspace}, hold also in the space $\mathcal{H}(d)$ of $xw$-symmetric polynomials. In particular, recall $\Ovd$ is $xw$-symmetric to justify Lemma~\ref{ovoidspace}.
\end{rmk}

\subsection{Constructions of \ctf\ curves}\label{constrCTF}
Now we turn to lower bounds on degrees of hypersurfaces that guarantee the existence of \ctf\ curves. For arbitrary $q$ we use the notion of a \textbf{blocking set}, i.e. a set $\mathcal{B}\subset\Ovd(\Fq)$ such that every circle in $\Ovd$ contains at least one point of $\mathcal{B}$ in order to construct extremely singular curves. For even $q$ we can apply ideas about partial spreads in $\Pp^3$ to construct \ctf\ curves without enforcing singularities.

A curve $C_f$ with a singularity at every point of a blocking set must be \ctf, since every circle of $\Ovd$ passes through a singularity of $C_f$ and is thus tangent to $C_f$. To be as efficient as possible, though, we prefer to use small blocking sets. From \cite{blocking}, we have the following.

\begin{lem}[\cite{blocking}, Theorem 2.2]\label{blockingset}
    There exists a blocking set $\mathcal{B}\subset\mathcal{O}(\Fq)$ with size
    \[
        |\mathcal{B}|\leq q^2+1-\sqrt[q+1]{\binom{q^2}{q}(q-1)!}.
    \]
\end{lem}

Now we can construct (very singular) \ctf\ curves.

\begin{thm}\label{invsingdeg}
    For $d > -1 + \sqrt{3\left(q^2+1-\sqrt[q+1]{\binom{q^2}{q}(q-1)!}\right)}$, there exists a polynomial $f$ of degree $d$ such that $C_f$ is \ctf.
\end{thm}

\begin{proof}
    Fix a blocking set $\mathcal{B}\subset\Ovd(\Fq)$. We consider the space of homogeneous polynomials of degree $d$ and the subspace of those that have a singularity at each point $P\in \mathcal{B}$. From Proposition~\ref{codims}, the space of solutions to this system thus has dimension at least
    \[
        \binom{d+3}{3} - 3|\mathcal{B}|.
    \]

    Of course we still want to avoid polynomials that are multiples of the quadratic $\varphi$ defining $\Ovd$. From Lemma ~\ref{ovoidspace} the dimension of the space of these polynomials is $\binom{d+1}{3}$. Not all such polynomials necessarily define surfaces tangent to $\Ovd$ at every $P\in B$, but this does not cause a difficulty. A surface $S=S_f$ of degree $d$ that does not contain $\Ovd$ but where $C_f$ is singular at every $P\in B$ is guaranteed to exist if
    \[
        \binom{d+3}{3} - 3|\mathcal{B}| > \binom{d+1}{3}.
    \]
    Since $\binom{d+3}{3}-\binom{d+1}{3} = (d+1)^2$, this amounts to saying
    \[
        d > -1 + \sqrt{3|\mathcal{B}|}.
    \]

    If $\mathcal{B}$ is from Lemma~\ref{blockingset} with size $|\mathcal{B}|\leq q^2+1-\sqrt[q+1]{\binom{q^2}{q}(q-1)!}$, then the degree bound above is satisfied by hypothesis. Thus a degree $d$ polynomial that defines a \ctf\ curve certainly exists.
\end{proof}

\begin{rmk}\label{inv-simp}
    Numerical data and limits computed in Mathematica suggest that
    \[
        q^2+1-\sqrt[q+1]{\binom{q^2}{q}(q-1)!} \leq 3q\ln(q).
    \]
    If this inequality holds then we may simplify (if weaken) the statement of Lemma~\ref{blockingset} to say there is a blocking set of size at most $3q\ln(q)$, and the statement of Theorem~\ref{invsingdeg} to say degree $d$ polynomials exist defining \ctf\ curves for all $d>-1+3\sqrt{q\ln(q)}$.
\end{rmk}

Another approach to constructing \ctf\ curves is to specify an $\Fq$-point on each circle at which a curve shall be tangent to that circle. Fortunately, once a curve is tangent to one circle at one $\Fq$-point, it is automatically tangent to $q-1$ more circles at the same point by Lemma~\ref{bundletan}. Towards this end, we will need to discuss \textbf{partial spreads}, collections $\mathcal{L}$ of lines in $\Pp^3$ that are pairwise disjoint.

\begin{lem}\label{spread-ctf}
    Consider a polynomial $f$ that defines a curve $C_f$ such that $\Ovd(\Fq)\subset C_f$. At each $P\in\Ovd(\Fq)$, choose a line $L_P \subset T_PC_f$ (there is only one choice if $C_f$ is smooth). If the collection $\mathcal{L} = \{L_P\mid P\in\Ovd(\Fq)\}$ is a partial spread of $\Pp^3$ then $C_f$ is \ctf.
\end{lem}

\begin{proof}
    At each $P\in\Ovd(\Fq)$, $C_f$ is tangent to $q$ circles by Lemma~\ref{bundletan}. If the lines of $\mathcal{L}$ form a partial spread of $\Pp^3$ then no plane contains more than one of these lines (else they would intersect). Thus no circle in $\Ovd$ tangent to $L_P$ is also to $L_Q$ for $Q\neq P$.
    Thus the $q^2+1$ points of $\Ovd(\Fq)$ contribute a total of $q(q^2+1)$ circles tangent to the lines of $\mathcal{L}$ and thus tangent to $C_f$. These are all the circles in $\Ovd$, so $C_f$ is \ctf.
\end{proof}

Thas' paper \cite{spread} on ovoidal translation planes gives us the following. 

\begin{lem}[\cite{spread}]\label{partspr}
    There exists a partial spread $\mathcal{L}$ of $\Pp^3$ consisting of $q^2+1$ lines such that each line in $\mathcal{L}$ is tangent to $\Ovd$ (necessarily at distinct $\Fq$-points) if and only if $q$ is even.
\end{lem}

When $q$ is even, the partial spreads guaranteed by Lemma~\ref{partspr} contain one line tangent at each $\Fq$-point of $\Ovd$, since there are $q^2+1$ points. We can construct \ctf\ curves by imposing the tangencies suggested by the partial spread.

\begin{thm}\label{linctf}
    For $d > -1 + \sqrt{2(q^2+1)}$ and $q$ even, there exists a polynomial $f$ of degree $d$ such that $C_f$ is \ctf.
\end{thm}

\begin{proof}
    Let $\mathcal{L}$ be a partial spread of $\Pp^3$ guaranteed by Lemma~\ref{partspr}. For each $P\in\Ovd(\Fq)$, let $L_P\in\mathcal{L}$ be the line tangent to $\Ovd$ at $P$.
    
    If we construct a hypersurface $S=S_f$ with $f(P)=0$ for each $P\in\Ovd(\Fq)$ then we will have $\Ovd(\Fq)\subset C_f$. If, in addition, $S_f$ is tangent to $L_P$ at $P$ for each $P\in \Ovd(\Fq)$, then the collection of lines tangent to $C_f$ at the points $P\in\Ovd(\Fq)$ will be precisely the spread $\mathcal{L}$. Then by Lemma~\ref{spread-ctf} the curve $C_f$ will be \ctf.

    The result of Proposition~\ref{codims} is that our $q^2+1$ tangency conditions above amount to at most $2(q^2+1)$ linear conditions in the space of polynomials of degree $d$. In the space of polynomials of degree $d$, the dimension of those polynomials satisfying the $2(q^2+1)$ (not necessarily independent) conditions is at least
    \[
        \binom{d+3}{3} - 2(q^2+1).
    \]
    Of course, all polynomials with the quadratic $\varphi$ as a factor are counted within these solutions, but we do not wish to consider them. So a polynomial $f$ of degree $d$ does exist making $C_f$ \ctf\ without $\Ovd\subset S_f$ if
    \[
        \binom{d+3}{3}-2(q^2+1) > \binom{d+1}{3}.
    \]
    Since $\binom{d+3}{3}-\binom{d+1}{3} = (d+1)^2$, this amounts to saying
    \[
        (d+1)^2>2(q^2+1),
    \]
    which is equivalent to the degree assumption in the Theorem statement.
\end{proof}

Note that Theorem~\ref{linctf} guarantees \ctf\ curves of any degree $d > -1 + \sqrt{2(q^2+1)}.$ In particular, we can find \ctf\ curves of any degree
\[
    d>\sqrt{2}q.
\]
This degree bound is linear in the size of the field, analogous to the degree $q+2$ curves in $\Pp^2$ discussed in \cite{transverse-free}. Unfortunately, the curves we have constructed here are not guaranteed to be smooth.

We do not have sufficiently powerful tools here to construct \emph{smooth} \ctf\ curves, but a start may be the following. Instead of merely stipulating that $C_f$ is tangent to each $L_P$ at $P$, we could choose a non-zero value for $\nabla f(P)$ to enforce tangencies while avoiding singularities at these points. Such a procedure can be accomplished e.g. using Lemma 2.1 of \cite{poonen}. We do not go into the technical details here, but simply remark that the degree bound becomes quadratic in $q$, much weaker than the linear bounds derived above.

\section{Curves in the hyperbolic plane}\label{sec:hyp}

Using the algebra of $xw$-symmetric curves introduced in Section~\ref{hyperbolic algebra}, we can begin to search for \ltf\ hyperbolic curves in analogy with our investigations into \ctf\ inversive curves from Section~\ref{sec:ctf}. Using results of subsection~\ref{curve conditions} we can set up a system of equations for the coefficients of an $xw$-symmetric polynomial to enforce tangency or singularity conditions.

Our two approaches in this section mirror the two approaches taken in Section~\ref{constrCTF}, though we present them here in the opposite order. In subsection~\ref{hyper ltf} we use the same tools from Theorem~\ref{linctf} to produce \ltf\ curves. Then in subsection~\ref{hyper blocking} we must first construct blocking sets in the hyperbolic plane before modifying Theorem~\ref{invsingdeg} to produce highly singular \ltf\ hyperbolic curves.

\subsection{Hyperbolic \ltf\ curves with tangency conditions}\label{hyper ltf}

Using the same partial spreads from Lemma~\ref{partspr} we can find tangencies to some of the $xw$-symmetric circles in $\Ovd$ at $\Fq$-points. Because there are more lines, $q^2-1$, than points, $\frac{q^2-q}{2}$, in the hyperbolic plane, a smooth \ltf\ hyperbolic curve must have tangencies at some non-$\Fq$ points, unlike the \ctf\ curves of Theorem~\ref{linctf}.

\begin{thm}\label{hypltf}
    For $d > -3/2+\sqrt{14q^2+10q-95/4}$, there exists a hypersurface $S$ of degree $d$ defining a \ltf\ curve in the hyperbolic plane.
\end{thm}

\begin{proof}
    We will use the same strategy as in Theorem \ref{linctf}. From Lemma~\ref{partspr} we have a partial spread $\mathcal{L}$ of $\Pp^3$ consisting of $q^2+1$ lines tangent to $|Ovd$. For each $P\in\Ovd(\Fq)$ there is a line $L_P\in\mathcal{L}$ tangent to $\Ovd$ at $P$.
    
    For each of the $(q^2-q)/2$ pairs $\{P,P^*\}$ in $\Ovd(\Fq)\backslash B$ defining points of a hyperbolic plane, choose one point in the pair, $P$. There is a unique $xw$-symmetric circle in $\Ovd$ tangent to $L_P$ at $P$, thus defining a hyperbolic line $\ell_P$. This of course only accounts for $(q^2-q)/2$ of the $q^2-1$ hyperbolic lines. For the remaining $(q^2+q-2)/2$ $xw$-symmetric circles $\ell$, we must select a non-$\Fq$ point $P_\ell$ on the circle.

    To be explicit about finding non-$\Fq$ points of low degree on circles, use the parametric description of the points of $\Ovd$ from subsection~\ref{sec:inv}, $(s^2+st+\gamma t^2:s:t:1)$. We can safely ignore $(1:0:0:0)$ as it is certainly an $\Fq$-point. For a circle defined by the linear function $Ax+By+Cz+Dw$, we search for solutions to
    \[
        A(x^2+st+\gamma t^2)+Bs+Ct+D=0.
    \]
    For each choice of $t\in\mathbb{F}_{q^r}$, we can solve for $s\in\mathbb{F}_{q^{2r}}$. To avoid the possibility of having both $s$ and $t$ in $\Fq$, we may simply choose $s\in\mathbb{F}_{q^r}\backslash\Fq$.

    Two planes in $\Pp^3$ intersect in a line, and a line may only intersect $\Ovd$ in at most 2 points. If we choose a point $P=(p_0:p_1:p_2:p_3)\in\Ovd\cap L$ defined over $\mathbb{F}_{q^r}$, then its Galois conjugates $(p_0^q:p_1^q:p_2^q:p_3^q),\ldots,(p_0^{q^{r-1}}:p_1^{q^{r-1}}:p_2^{q^{r-1}}:p_3^{q^{r-1}})$ must all be in $\Ovd\cap L$ which can thus only happen for $r=2$. That is, two circles in $\Ovd$ cannot share a point defined over $\mathbb{F}_{q^r}$ for $r>2$.

    Choosing $s\in\mathbb{F}_{q^3}\backslash\mathbb{F}_{q^2}$, we find a distinct point $P_\ell$ on each $\ell$ defined over $\mathbb{F}_{q^6}$

    Within the space $\mathcal{H}(d)$ of degree $d$ $xw$-symmetric polynomials from subsection~\ref{curve conditions}, consider the conditions imposed by demanding $C_f$ be tangent to each $\ell_P$ at $P$ for the first $(q^2-q)/2$ hyperbolic lines, and tangent to $\ell$ at $P_\ell$ for the remaining $(q^2+q-2)/2$ hyperbolic lines. From Lemma~\ref{tang-cond}, each of the $(q^2-q)/2$ $\Fq$-points contributes 2 linear constraints, for a total of $q^2-q$. From Lemma~\ref{deg-r}, each of the remaining $(q^2+q-2)/2$ points $P$ defined over $\mathbb{F}_{q^6}$ contribute 12 constraints, for a total of $6(q^2+q-2)$. Thus Proposition~\ref{codims} says the space of $xw$-symmetric polynomials of degree $d$ satisfying these constraints has dimension at least
    \[
        dim(\mathcal{H}(d)) - (q^2-q) - 6(q^2+q-2) = dim(\mathcal{H}(d)) - (7q^2+5q-12).
    \]

    But since $\Ovd$ is $xw$-symmetric, multiplying its defining polynomial $\varphi$ by any degree $d-2$ $xw$-symmetric polynomial yields a degree $d$ $xw$-symmetric surface that contains $\Ovd$. To avoid this, we simply need
    \[
        dim(\mathcal{H}(d)) - (7q^2+5q-12) > dim(\mathcal{H}(d-2)).
    \]
    Fortunately, whether $d$ is even or odd, $dim(\mathcal{H}(d))-dim(\mathcal{H}(d-2)) = (d^2+3d+2)/2$. Thus a solution to our problem exists as long as
    \[
        d^2+3d+2 > 14q^2+10q-24.
    \]
    The larger root (in $d$) to the quadratic $d^2+3d-(14q^2+10q-26)$ is $-3/2+\sqrt{14q^2+10q-95/4}$, so the hypothesis of the theorem suffices to produce polynomials that define \ltf\ hyperbolic curves.
\end{proof}

Next, we will find an analogous result using the familiar idea of placing a singularity at every point of a blocking set. Thus we must first explore blocking sets in hyperbolic planes.

\subsection{Blocking sets in hyperbolic planes}\label{hyper blocking}

We will first find an upper bound on the size of a blocking set in finite hyperbolic planes, which we will then use to achieve a bound on degrees of \ltf\ hyperbolic curves. We directly modify an argument given in \cite{blocking}, first finding a lower bound for the maximal number of lines which are hit by an $n$-element set of points in a finite hyperbolic plane. 
\begin{defn}
    Given a set $X$ of points in a hyperbolic plane $H$ of even order $q$, define the blocking number of $X$ by
    \[
        \beta(X)\colonequals|\{L : L\cap X\neq\emptyset\}|.
    \]
    We say a line $L$with $L\cap X\neq\emptyset$ is blocked by $X$.
\end{defn}

\begin{lem}\label{hyp-block-stat}
    For every $0\leq n\leq q(q-2)/2$ there exists an $n$ element set of points $X$ such that
    \[
        \beta(X)\geq (q^2-1)\bigg(1-\frac{\binom{q(q-1)/2-n}{q/2}}{\binom{q(q-1)/2}{q/2}}\bigg).
    \]
\end{lem}

\begin{proof}
    We inductively construct the desired point set $X$, first observing that the claim clearly holds for $n=0,X=\emptyset$. Now if $X$ is an $n$-element set with $n<q(q-2)/2$ satisfying $\beta(X)\geq(q^2-1)\bigg(1-\frac{\binom{q(q-1)/2-n}{q/2}}{\binom{q(q-1)/2}{q/2}}\bigg)$,  we claim there is a point $P\not\in X$ that is contained in at most $\frac{\beta(X)q/2-n(q+1)}{q(q-1)/2-n}$ lines which are already blocked by $X$. Adding this point to $X$ will complete the inductive construction of the sets we desire.
    
    Consider the set of pairs of points and lines
    \[
        Q_X:=\{(P,L) : P\in L, L\cap X\neq\emptyset\}.
    \]
    On the one hand, the cardinality of this set is just the number of lines blocked by $X$ times the number of points on a hyperbolic line, i.e. $\beta(X)q/2$. On the other hand, we have
    \[
        |Q_X|=\sum_{P\in X}(q+1)+\sum_{P\notin X}|\{L : P\in L, L\cap X\neq\emptyset\}|.
    \]
    
    We can now compute the average, over $P\not\in X$, of the number of lines through $P$ blocked by $X$. This average is given by
    \[
        \frac{1}{q(q-1)/2-n}\sum_{P\notin X}|\{L : P\in L, L\cap X\neq\emptyset\}|=\frac{\beta(X)q/2-n(q+1)}{q(q-1)/2-n},
    \]
    and hence there must exist a point $P$ contained in at most $\frac{\beta(X)q/2-n(q+1)}{q(q-1)/2-n}$ blocked lines as we claimed. If we add this point to $X$, we obtain
    \[
        \beta(X\cup\{P\})\geq \beta(X)+(q+1)-\frac{\beta(X)q/2-n(q+1)}{q(q-1)/2-n}.
    \]
    Applying the hypothesis $\beta(X)\geq(q^2-1)\bigg(1-\frac{\binom{q(q-1)/2-n}{q/2}}{\binom{q(q-1)/2}{q/2}}\bigg)$ and the identity $\binom{a}{b} = \frac{a}{a-b}\binom{a-1}{b}$, we can see
    \[
        \beta(X\cup\{P\}) \geq \frac{q^2-2q-2n}{q^2-q-2n}\beta(X) + (q+1) + \frac{n(q+1)}{q(q-1)/2} \geq (q^2-1)\left(1-\frac{\binom{q(q-1)/2-(n+1)}{q/2}}{\binom{q(q-1)/2}{q/2}}\right).
    \]
\end{proof}

This allows us to bound the size of a blocking set in the hyperbolic plane.

\begin{thm}\label{hypblock}
    In a hyperbolic plane of order $q$, there exists a blocking set $X$ of size
    \[
        |X| \leq \left\lceil q(q-1)/2-\sqrt[q/2]{\binom{q(q-1)/2}{q/2}\frac{(q/2)!}{q^2-1}} \right\rceil
    \]
\end{thm}

\begin{proof}
    Defining $g(n):=(q^2-1)(1-\frac{\binom{q(q-1)/2-n}{q/2}}{\binom{q(q-1)/2}{q/2}}),$ we have to find $n$ such that $g(n)\geq q^2-1$ or equivalently $g(n)>q^2-2$. Expanding this last inequality and using the inequality
    \[
        \binom{q(q-1)/2-n}{q/2}\leq \frac{(q(q-1)/2-n)^{q/2}}{(q/2)!},
    \]
    we find it suffices to have
    \[
        n > q(q-1)/2-\sqrt[q/2]{\binom{q(q-1)/2}{q/2}\frac{(q/2)!}{q^2-1}}.
    \]
    Let $n$ be the ceiling of this last expression, which is less than $q(q-2)/2$. Then Lemma~\ref{hyp-block-stat} implies there is an $n$-element set $X$ with $\beta(X)\geq g(n)\geq q^2-1$, so $X$ is blocking.
\end{proof}
We can now use this blocking set to create a hyperbolic \ltf\ curve similarly to Theorem~\ref{invsingdeg}.

\begin{thm}\label{hypsingdeg}
    For all $d>-3+\sqrt{12q^2-12q+1-24\sqrt[q/2]{\binom{q(q-1)/2}{q/2}\frac{(q/2)!}{q^2-1}}}$ there exists an $xw$-symmetric polynomial $f$ of degree $d$ defining an \ltf\ hyperbolic curve $C_f$.
\end{thm}

\begin{proof}
    We create a hyperbolic-line-transverse-free curve by enforcing a singularity at each point of a hyperbolic blocking set $\mathcal{B}$. As discussed in Remark~\ref{same for hyp}, each singularity incurs 3 linear constraints on the coefficients of a polynomial. As in all previous such arguments, we must also ensure that the space of polynomials satisfying our singularity conditions also includes polynomials $f$ so that $\Ovd\not\subset S_f$. In the space $\mathcal{H}(d)$ of degree $d$, $xw$-symmetric polynomials, there are some satisfying all of our constraints as long as
    \[
        dim(\mathcal{H}(d)) - 3|B| > dim(\mathcal{H}(d-2).
    \]
    As in the proof of Theorem~\ref{hypltf}, $dim(\mathcal{H}(d))-dim(\mathcal{H}(d-2)) = (d^2+3d+2)/2$, so it suffices to have
    \[
        (d^2+3d+2)/2 - 3|\mathcal{B}|>0.
    \]

    The above inequality is certainly satisfied if $d>-3+\sqrt{1+24|\mathcal{B}|}$. Using a blocking set $\mathcal{B}$ from Theorem~\ref{hypblock}, we find a hyperbolic \ltf\ curve of degree
    \[
        d>-3+\sqrt{12q^2-12q+1-24\sqrt[q/2]{\binom{q(q-1)/2}{q/2}\frac{(q/2)!}{q^2-1}}}
    \]
    as yearned for.
\end{proof}

\begin{rmk}\label{hyp-simp}
    Numerical data and limits computed in Mathematica suggests that
    \[
        q(q-1)/2-\sqrt[q/2]{\binom{q(q-1)/2}{q/2}\frac{(q/2)!}{q^2-1}} \leq 2q\ln(q).
    \]
    If this inequality holds then we may simplify (if weaken) the statement of Theorem~\ref{hypblock} to say there is a blocking set of size at most $\lceil 2q\ln(q)\rceil$, and the statement of Theorem~\ref{hypsingdeg} to say degree $d$ polynomials exist defining \ltf\ hyperbolic curves for all $d>-3+\sqrt{1+48q\ln(q)}$.
\end{rmk}

\bibliographystyle{plain}

\end{document}